\documentclass[12pt]{article} 

\usepackage{amsmath,amssymb}
\usepackage[english]{babel}
\usepackage{bbm, amsthm}

\usepackage{titlesec}
\titleformat*{\section}{\large \bfseries}
\titlespacing{\section}{0pt}{\parskip}{\parskip}

\titleformat*{\subsection}{\normalsize \bfseries}
\titlespacing{\subsection}{0pt}{\parskip}{\parskip}

\usepackage{setspace}

\numberwithin{equation}{section}

\newtheorem{theorem}{Theorem}[section]
\newtheorem{lemma}[theorem]{Lemma}
\newtheorem{corollary}[theorem]{Corollary}
\newtheorem{proposition}[theorem]{Proposition}
\newtheorem{definition}[theorem]{Definition}

\newtheorem{remark}[theorem]{Remark}

\def\cS{\mathcal{S}}

\def\P{\mathbb{P}}
\def\Y{\mathbf{Y}}

\DeclareMathOperator*{\argmin}{arg\,min}

\title{\large{On the regularity of American options with regime-switching uncertainty}}

\author{Saul D. Jacka\thanks{Department of Statistics, University Of Warwick and 
The Alan Turing Institute, British Library, 96 Euston Road, London NW1 2DB London, UK (s.d.jacka@warwick.ac.uk).}
\and Adriana Ocejo\thanks{Department of Mathematics and Statistics, University of North Carolina at Charlotte (amonge2@uncc.edu).}}

\date{June 2017}

\begin{document}

\maketitle

\begin{abstract}
We study the regularity of the stochastic representation of the solution of a class of initial-boundary value problems related to a regime-switching diffusion.
This representation is related to the value function of a finite-horizon optimal stopping problem such as the price of an American-style option in finance.
We show continuity and smoothness of the value function using coupling and time-change techniques.
As an application, we find the minimal payoff scenario for the holder of an American-style option in the presence of regime-switching uncertainty under the assumption that the transition rates are known to lie within level-dependent compact sets.
\end{abstract}



\section{Introduction}

\noindent
\label{Intro}

Let $B=(B_t)_{t\geq 0}$ be a Brownian motion and
$Y=(Y_t)_{t\geq 0}$ be a continuous-time finite-state Markov chain, with respect to
a common filtered probability space $(\Omega,\mathcal{F},(\mathcal{F}_t)_{t\geq 0}, P)$,
where $(\mathcal{F}_t)_{t\geq 0}$ satisfies the usual conditions.
Note that  Lemma 2.5 of \cite{JM} tells us that $B$ and $Y$ are independent.

Let $\mathcal{S}=\{1,2,\ldots,m\}$ denote the state space of $Y$ and $\pi=(\pi[i,j])$ its $Q$-matrix so that
\[
\pi[i,j]\geq 0, \quad \mbox{for $i\neq j$ \quad and} \quad \sum_{j=1}^m \pi[i,j]=0 \quad \mbox{for $i=1,2,\ldots,m$.}
\]

Suppose that the process $X=(X_t)_{t\geq 0}$ obeys
the stochastic differential equation with regime-switching
\begin{equation}\label{eq:dynX}
X_t=x+\int_0^t a(X_s,Y_s)dB_s+\int_0^t\mu(X_s,Y_s)ds, \qquad x\in \mathbb{R},
\end{equation}
where, for each $y\in\cS$, $a(\cdot,y)$ and $\mu(\cdot,y)$ are, locally Lipschitz continuous  on the state space of $X$ and $a$ is {\em positive}.
Denote by $\mathbb{L}^\pi$ the operator related to the generator of the Markov process $(X,Y)$, given by
\begin{equation}\label{bigL}
\begin{split}
\mathbb{L}^\pi w(x,y,t) = \frac{1}{2} a^2(x,y)w_{xx}(x,y,t)&+\mu(x,y)w_x(x,y,t)-w_t(x,y,t) \\
   & \hspace{-1cm} + \sum_{y'\in \mathcal{S}, y'\neq y}[w(x,y',t)-w(x,y,t)]\pi[y,y'].
\end{split}
\end{equation}
For a given rate matrix $\pi$, consider the value of the optimal stopping problem with finite time horizon $T>0$ and regime-switching associated with $(X,Y)$:
\begin{equation} \label{eq:valueOSPintro}
v(x,y,t)=\sup_{\tau\leq t}\, E_{x,y}(e^{-\alpha \tau} g(X_\tau)), \qquad (x,y,t)\in \mathbb{R}\times \mathcal{S}\times [0,T],
\end{equation}
where
$\alpha\geq 0$ and $g:\mathbb{R} \rightarrow [0,\infty)$ is assumed to be a
$\beta$-H\"older continuous function for some $0<\beta\leq 1$.
We write $E_{x,y}$ to denote the expectation conditioned on $(X_0=x, Y_0=y)$.

We are primarily interested in analytical properties of the value function in (\ref{eq:valueOSPintro}).
In particular, we will show that for each $y\in \mathcal{S}$ the function
$v(\cdot,y,\cdot):\mathbb{R}\times[0,T]\rightarrow \mathbb{R}$ is $\beta/2$-H\"older continuous
and $v$ solves the initial-boundary value problem
\begin{equation} \label{eq:IBP}
\begin{array}{rll}
(\mathbb{L}^\pi-\alpha\, )  v(x,y,t) & =0,      & \mbox{in}\; \mathcal{C} \\
                            v(x,y,0) & =g(x),   & \mbox{in}\; \mathbb{R}\times \mathcal{S}\times\{0\} \\
                            v(x,y,t) & = g(x),  & \mbox{on}\; \partial \mathcal{C}
\end{array}
\end{equation}
where $\mathcal{C}=\{(x,y,t)\in \mathbb{R}\times \mathcal{S}\times(0,T]:v(x,y,t)>g(x)\}$.
This in turn yields that $v(\cdot,y,\cdot)$ is of the class $C^{2,1}$ in the set
\[
\mathcal{C}_y=\{(x,t)\in \mathbb{R}\times (0,T]:\, v(x,y,t)>g(x)\}.
\]

The stochastic representation of the solution of a problem of the form in (\ref{eq:IBP}) and
in the setting where $X$ is a diffusion without regime-switching is of course very well-known (see \cite{Friedman1975})
and the relation to an optimal stopping problem is standard (\cite{Oks}, \cite{PS}).
This relationship allows the use of PDE methods to tackle the latter problem \cite{PS}, such as finding a solution to an American option problem.
However, in order to establish the desired connection, one typically requires continuity of the value function $v$.
To this end, a number of subtle regularity conditions on the parameters of the problem must be imposed.
In Section \ref{sec:cont} we show that $v(\cdot,y,\cdot)$ \textit{is} locally $\beta/2$-H\"older continuous (Theorem \ref{thm:cont} below).
Our results generalise those of Fleming and Soner \cite{FS} and Byraktar, Song and Yang \cite {BSY} in the context of optimal stopping (although they deal with combined control and stopping), since they assume uniformly Lipschitz coefficients and payoff and do not allow the jumps present in our model.

In the context of regime-switching diffusions, an early explicit example is Di Masi, Kabanov and Runggaldier \cite{MKR}, where they consider option pricing in an incomplete market with regime switching. More recently Baran, Yin and Zhu \cite{BYZ2013} studied the stochastic representation
for a generic initial-boundary value problem generalizing the results by Friedman \cite{Friedman1975}.  
A standing assumption in their theory is that the PDE problem is defined on an open and bounded domain on the state space of $X$.
In practice, such domains may be unbounded but we can get around this issue by a local argument which requires some continuity of the underlying
value function (see proof of Theorem \ref{thm:regThm}).
Continuity of $v$ also plays a central role for instance, in the derivation of optimal stopping rules and to determine the shape of the optimal stopping boundary.
In general, so-called tangency problems (see \cite{FS}) may prevent continuity of the value function, but in our model this is precluded by a local ellipticity assumption (see the proof of Proposition \ref{prop:locallyLipinx}).

The study of the solution of an American-style option problem of the form in (\ref{eq:valueOSPintro})
has been addressed in the literature under the special case of Markov-modulated geometric Brownian motion dynamics $dS_t=S_t(\sigma(Y_t)dB_t+\mu(Y_t)dt)$.
For instance, Buffington and Elliott \cite{BE2002} analyze the American put, with $g(x)=(K-x)^+$, and discuss the determination of the value function and the shape of the optimal stopping boundary by means of a direct application of It\^o's formula.
More recently, Le and Wang \cite{LeWang} 
studied analytical properties of the value function of an American-style option under the condition that $g$ belongs
to the class of non-negative, non-increasing, convex functions with bounded support and twice differentiable on the support.

In Sections \ref{sec:cont} and \ref{sec:smooth}, we study continuity and smoothness of $v$ and we only assume that $g$ is H\"older continuous. 
Our approach is of interest because it uses  coupling arguments as well as classical path properties of Brownian motion,
which are purely probabilistic tools.
The smoothness results ensure that standard numerical schemes for solving the optimal stopping problem are stable.

Throughout the paper we make the following standing assumption: 
\begin{center}\textbf { (A1) }
$E_{x,y} \left(\sup_{0\leq t\leq T} (|a(X_t,Y_t)|+|\mu(X_t,Y_t)|)\,\right) \leq N$
for some $N=N(x,y,T)<\infty$ and $N$ is continuous as a function of $x$.
\end{center}
Notice that if $a(\cdot,y)$ and $\mu(\cdot,y)$ satisfy linear growth conditions for each $y$ then (A1)
follows by results on estimates of the moments of regime-switching diffusions (see \ref{app:estimates}).

In financial applications, regime-switching processes have been used to better reproduce asset price behavior from market data.
For example, they can generate a volatility smile \cite{YZZ2006}, which is impossible under a constant volatility.
Commodity prices such as electricity are prone to spikes \cite{Carmona}, and modeling based on regime-switching coefficients has been proposed \cite{Mountetal}.
A similar effect of dramatic price rises and crashes appears in so-called asset price bubbles. Bubbles have been addressed as following a regime-switching structure \cite{bubbles}
and characterized as strict local martingales (see \cite{CoxHobson}, \cite{JPS2010}, \cite{Protter2013} and references therein).

Typically, the transition rates of the Markov chain are specified via a constant $Q$-matrix.
However, empirical studies suggest that time-varying transition and path-dependent rates
improve the forecasting ability of the phases of an economy such as expansions,
contractions and duration of the regimes (see \cite{Diebold}, \cite{Filardo}, \cite{Kanamuraetal} and references therein).

In Section \ref{sec:extsce}, we use the results of the foregoing sections to find the minimal payoff scenario for the holder of an American-style option
who only knows level-dependent bounds on the transition rates of the Markov chain.
More precisely, we let $\pi$ denote an {\it admissible} time-varying, $(\mathcal{F}_t)$-adapted rate matrix $\pi=(\pi_t)_{t\geq 0}$ such that for each $t\geq 0$,
the $Q$-matrix $\pi_t=(\pi_t[i,j])$ satisfies
\begin{equation} \label{eq:constraint}
\pi_t[i,i+z]\in A_{i,z}^+,\quad \pi_t[i,i-z]\in A_{i,z}^-, \qquad t\geq 0
\end{equation}
where $A_{i,z}^+,A_{i,z}^-$ are compact subsets of $(0,\infty)$.
Denote by $\mathcal{A}$ the set of all admissible rate matrices.
We aim to find, for each initial condition $(x,y,t)$, an admissible rate matrix $\pi^l$ that attains the infimum
\begin{equation} \label{eq:infvalue}
V^l(x,y,t)=\inf_{\pi \in \mathcal{A}}\,v(x,y,t;\pi).
\end{equation}
Here, $v(x,y,t;\pi)$ is as in (\ref{eq:valueOSPintro}) and the notation is to emphasize the dependence on $\pi$.

\section{Continuity of $v$} \label{sec:cont}

First we show that the function $v(x,y,\cdot)$ is $\beta/2$-H\"older continuous on $[0,T]$ over a neighborhood of $x$.
This is where Assumption (A1) is used.
Next, we will see that $v(\cdot,y,t)$ is locally Lipschitz continuous in the set $\{x:v(x,y,t)>g(x)\}$ via path properties of Brownian motion.
As a consequence, it is seen that these properties together imply the main result of this section, Theorem \ref{thm:cont}.

\pagebreak
\begin{proposition} \label{prop:holderInTime}
For each $(x,y)\in\mathbb{R}\times \mathcal{S}$, the function $v(x,y,\cdot)$ is $\beta/2$-H\"older continuous uniformly over a neighborhood of $x$.
\end{proposition}
\begin{proof}
Fix $(x,y)\in \mathbb{R}\times\mathcal{S}$. Take $C>0$  such that $|g(z)-g(z')|\leq C|z-z'|^\beta$ for all $z,z'\in \mathbb{R}$, with $0<\beta\leq 1$.

Let $0\leq t_1<t_2\leq T$ and take an arbitrary $\epsilon>0$.
Suppose that $\tau_2$ is an $\epsilon$-optimal stopping time for the problem $v(x,y,t_2)$ and set $\tau_1=\tau_2\wedge t_1$ so that $\tau_1$ is suboptimal for $v(x,y,t_1)$.
Then $\tau_1 \leq \tau_2\leq t_2$ and since $v(x,y,\cdot)$ must be an increasing function of time we have, defining the martingale $M$ by $M_t=\int_0^t a(X_s,Y_s)dB_s$,
\begin{eqnarray}\label{Hol}
\begin{aligned}
0   & \leq v(x,y,t_2)-v(x,y,t_1) \leq \epsilon+ E_{x,y}( e^{-\alpha \tau_2}g(X_{\tau_2}))-E_{x,y}(e^{-\alpha \tau_1} g(X_{\tau_1}))\\
    & \leq \epsilon+E_{x,y}( e^{-\alpha \tau_2}|g(X_{\tau_2})-g(X_{\tau_1})|)
     \leq \epsilon+C\,E_{x,y}(|X_{\tau_2}-X_{\tau_1}|^\beta) \\
    &\leq \epsilon+C\,E_{x,y}(|X_{\tau_2}-X_{t_1}|^\beta I_{\{t_1< \tau_2\}}) \\
    &\leq \epsilon+C'\,E_{x,y}\left(\left\{|M_{\tau_2}-M_{t_1}|^\beta +\left|\int_{t_1}^{\tau_2}\mu(X_s,Y_s)ds\right|^\beta \right\} I_{\{t_1<\tau_2\}}\right)\\
    &\leq  \epsilon+C'\,E_{x,y}\left(\sup_{t_1\leq t\leq t_2}|M_t-M_{t_1}|^\beta +\left[\sup_{t_1\leq t\leq t_2}|\mu(X_t,Y_t)|(t_2-t_1)\right]^\beta\right)\\
    &\leq \epsilon+C''\,E_{x,y}\left( [\langle M\rangle_{t_2}-\langle M\rangle_{t_1}]^{\frac{\beta}{2}}+\left[\sup_{t_1\leq t\leq t_2}|\mu(X_t,Y_t)|(t_2-t_1)\right]^\beta \right)\\
\end{aligned}
\end{eqnarray}
where the last inequality follows from the Burkholder-David-Gundy inequalities  \cite[Corollary IV.4.2]{Rev-Yor}, and constants vary from line to line.

Now
$$
\langle M\rangle_t=\int_0^ta^2(X_s,Y_s)ds
$$
so the last line of (\ref{Hol}) is dominated by
\begin{eqnarray}\label{Hol2}
\begin{aligned}
&\epsilon+C''\,E_{x,y}\left( \left[\int_{t_1}^{t_2}a^2(X_s,Y_s)ds\right]^{\frac{\beta}{2}}
    +\left[\sup_{t_1\leq t\leq t_2}|\mu(X_t,Y_t)|(t_2-t_1) \right]^\beta\right)\\
&\leq \epsilon
    +C''\,E_{x,y}\left( \left[\sup_{t_1\leq t\leq t_2}a^2(X_t,Y_t)(t_2-t_1)\right]^{\frac{\beta}{2}}
    +\left[\sup_{t_1\leq t\leq t_2}|\mu(X_t,Y_t)|(t_2-t_1) \right]^\beta\right).
\end{aligned}
\end{eqnarray}
Now taking $t_2\leq t_1+1$ without loss of generality, and noting that $\beta\leq 1$, we see that the last line in (\ref{Hol2}) is dominated by
$$
\epsilon+C'''\,E_{x,y}\left(\sup_{ t\leq T}(|a(X_t,Y_t)|+|\mu(X_t,Y_t)|)\right)(t_2-t_1)^{\frac{\beta}{2}}\leq \epsilon+C''' N(x,y,T)(t_2-t_1)^{\frac{\beta}{2}}, 
$$
where by (A1), $N=N(x,y,T)$ depends on $x,y$ and $T$ only and is continuous in $x$. 
Finally,  recalling that $\epsilon>0$ can be made arbitrarily small, we obtain
\begin{equation} \label{eq:Lipsconstant}
0\leq v(x,y,t_2)-v(x,y,t_1) \leq C'''(t_2-t_1)^{\beta/2} N,
\end{equation}
and conclude that $v$ is $\beta/2$-H\"older continuous in $t$, uniformly in a neighborhood of $x$.
\end{proof}

The proof of the next lemma uses standard techniques so we defer the proof to the \ref{app:proofs}.
We note that we only impose  continuity and boundedness from below as conditions on $g$  in this result and  Proposition \ref{prop:locallyLipinx} below.

\begin{lemma} \label{lemma:lsc}
Assume that $g$ is continuous. For each $(y,t)\in \mathcal{S}\times (0,T]$, the function $v(\cdot,y,t)$ is lower semi-continuous in $\mathbb{R}$.
\end{lemma}

We now establish that $v$ is locally Lipschitz continuous as a function of $x$ in
the continuation region $\mathcal{C}=\{ (x,y,t)\in \mathbb{R}\times \mathcal{S} \times (0,T]: \, v(x,y,t)>g(x)\}$ via coupling of stochastic processes and elementary properties of Brownian motion.
Let us introduce the following notation.
For each $(y,t)\in \mathcal{S} \times (0,T]$ the $(y,t)$-section of the continuation region is
\begin{equation} \label{eq:yt_sec}
\mathcal{C}_{y,t}=\{x\in \mathbb{R}:\, v(x,y,t)>g(x)\}.
\end{equation}
As a consequence of the lower semi-continuity of $v(\cdot,y,t)$, for $t> 0$, the $(y,t)$-section $C_{y,t}$ is an open subset of $\mathbb{R}$.
Notice that if $t=0$ then $v(x,y,0)=g(x)$ for all $(x,y)$ and so $C_{y,0}=\emptyset$.
Similarly for each $y\in\mathcal{S}$, the $y$-section is
\begin{equation} \label{eq:y_sec}
\mathcal{C}_y=\{(x,t)\in \mathbb{R}\times (0,T]:\,v(x,y,t)>g(x)\}.
\end{equation}

\begin{proposition} \label{prop:locallyLipinx}
Assume that $g$ is continuous.
For each $(y,t)\in \mathcal{S}\times (0,T]$, the function $v(\cdot,y,t)$ is locally Lipschitz  continuous in the section $C_{y,t}$.
\end{proposition}
\begin{proof}
Fix $(y,t)\in \mathcal{S}\times (0,T]$ and note that without loss of generality, we may assume that $C_{y,t}$ is non-empty.
Let $I$ be an open and bounded interval whose closure is in $C_{y,t}$.
Since $v$ is increasing in $t$, the open subsets $C_{y,t}$ increase in $t$ as well and we must have that
$C_{y,t_0}\subseteq C_{y,t}\subseteq C_{y,t_1}$ whenever $0<t_0<t<t_1\leq T$. Let $t_0<t$ be sufficiently large so that $I\subset C_{y,t_0}$.
In particular, for some $t_1<T$,
\[
R:=I\times (t_0,t_1)\subsetneq C_y.
\]

\pagebreak
Let $x,x'\in I$ and assume without loss of generality that $x>x'$ (the reverse case follows by symmetry). Suppose that $X_t=x+\int_0^t a(X_s,Y_s)dB_s+\int_0^t \mu(X_s,Y_s)ds$ starts from $x$
and let $X'_t =x'+\int_0^t a(X'_s,Y_s)dB'_s+\int_0^t \mu(X'_s,Y_s)ds$ be started from $x'$, with driving Brownian motion $B'=-B$.
Note that, by uniqueness in law of the solution to (\ref{eq:dynX}),
$$
v(x',y,t)=\sup_{\tau\leq t}E( e^{-\alpha\tau}g(X'_\tau)).
$$

Define $Z=(Z_s)_{s\geq 0}$ by $Z_s:=X_s-X_s^\prime$ so that $Z_s=r+M_s+A_s$ where $r=x-x'>0$ and
\begin{eqnarray}\label{eq:ctslocmtgM}
\begin{aligned}
M_s&=\int_0^s [\,a(X_u,Y_u)+a(X_u^\prime,Y_u)\,]dB_u, \\
A_s&=\int_0^s [\mu(X_u,Y_u)-\mu(X'_u,Y_u)]du, \quad s\geq 0.
\end{aligned}
\end{eqnarray}

Consider the coupling time $\tau(x,x'):=\inf\{s>0:\,Z_s\leq 0\}$ of $X$ and $X'$.
Let $t_2=t-t_0$.
The function $v(\cdot,y,t)$ is bounded on $I$ (as it is lower semi-continuous and finite), say by $K/2$.
Thus upon stopping $X_s$ and $X'_s$ at first exit from the interval $I$ it follows that
\begin{equation} \label{eq:differencepayoff}
|v(x,y,t)-v(x',y,t)| \leq K\, P(\,t_2<\tau(x,x')\,).
\end{equation}

We next show that $P(t_2<\tau(x,x'))$ is $O(x-x')$ as $r=x-x'\rightarrow 0$, which implies Lipschitz continuity on the interval $I$.
Since $a$ and $\mu$ are continuous, we must have that $a(\cdot,\cdot)\geq \delta>0$ and $\mu(\cdot,\cdot)\leq \eta$ on  $I\times \cS$. Setting $k(s)=4 \delta^2\,s$ and $\nu=\frac{\eta}{2\delta^2}$ we obtain
\[
\begin{split}
\langle M\rangle_s  & =\int_0^s [\,a(X_u)+a(X_u^\prime)\,]^2\sigma^2(Y_u)du \\
& \geq \int_0^s[\,2\min\{a(X_u,Y_u),a(X_u^\prime,Y_u)\}\,]^2du \geq k(s).
\end{split}
\]
and
$$
A_s\leq \eta s\leq \nu {\langle M\rangle_s}.
$$
Using the Dubins-Dambis-Schwarz Theorem  (see \cite[V.1.6]{Rev-Yor}), there is a standard Brownian motion $W=(W_s)_{s\geq 0}$
such that $M_s=W_{\langle M\rangle_s}$.
Moreover,  $\langle M\rangle_s$ has continuously increasing paths and so
\begin{eqnarray}\label{Lip}
\begin{aligned}
\{\, t_2 < \tau(x,x') \,\}
& \subseteq \{\, W_{\langle M\rangle_s}+\nu \langle M\rangle_s >-r,\; \forall\,s\leq t_2\, \}\\
&=\{\, W_s+\nu s>-r,\; \forall s\leq \langle M\rangle_{t_2} \,\}  \\
& \subseteq \{\,\inf_{s\leq \langle M\rangle_{t_2}}W_s+\nu s \geq -r\,\}
  \subseteq \{\,\inf_{s\leq k(t_2)}W_s+\nu s \geq -r\,\}.
\end{aligned}
\end{eqnarray}

It follows that (see e.~g. \cite{BS})
\[
\begin{aligned}
P(t_2<\tau(x,x')) & \leq P\left(\inf_{s\leq k(\tau_2)}W_s +\nu s\geq -r\right) \\
                  & \leq 1-\Phi\left(\frac{-r-\nu k(t_2)}{\sqrt {k(t_2)}}\right)-e^{-2\nu r}\Phi\left(\frac{-r+\nu k(t_2)}{\sqrt {k(t_2)}}\right)
\end{aligned}
\]
and the right-hand side is $O(r)$ as $r\rightarrow 0$, 
which concludes the proof.
\end{proof}

\begin{theorem} \label{thm:cont}
Consider the value function $v$ in (\ref{eq:valueOSPintro}).
For each $y\in \mathcal{S}$, the function $v(\cdot,y, \cdot):\mathbb{R}\times [0,T]\rightarrow \mathbb{R}$ is locally $\beta/2$-H\"older continuous.
\end{theorem}
\begin{proof}
Let $K$ be a compact subset of $\mathbb{R}\times [0,T]$ and $(x,t), (x',t')\in K$.
Fixing $y\in \mathcal{S}$, we shall write $v(x,t)\equiv v(x,y,t)$ and $v(x',t')\equiv v(x',y,t')$.

By Proposition \ref{prop:holderInTime}, $v(x',\cdot)$ is $\beta/2$-H\"older continuous. Recall that the H\"older constant in (\ref{eq:Lipsconstant}) depends on $x$ continuously. In particular, it is bounded in the compact set $K$ as a function of $x$, say by some $\bar{N}>0$. Then we have
\[
|v(x',t)-v(x',t')|\leq \bar{N}|t-t'|^{\beta/2}.
\]
Also, by Proposition \ref{prop:locallyLipinx}, $v(\cdot,t)$ is locally Lipschitz continuous in $C_{y,t}$
and $\beta$-H\"older continuous everywhere else (since $v=g$ outside $C_{y,t}$).
Thus $v(\cdot,t)$ is $\beta/2$-H\"older continuous on the bounded set $\{x\in \mathbb{R}:\,(x,t)\in K\}$,
and so for some $C=C(K)>0$ we have
\[
|v(x,t)-v(x',t)|\leq C|x-x'|^{\beta/2}.
\]
Setting $M=\max\{C,\bar{N}\}$, the triangle inequality yields
\[
\begin{split}
|v(x,t)-v(x',t')|  & \leq C |x-x'|^{\beta/2}+\bar{N}|t-t'|^{\beta/2} \\
                   & \leq M\, 2^{1-\beta/4}\,(\sqrt{|x-x'|^2+|t-t'|^2}\,)^{\beta/2}.
\end{split}
\]

This shows that $v$ is locally $\beta/2$-H\"older continuous.
\end{proof}

\section{Smoothness of $v$} \label{sec:smooth}

In this section we use the continuity of $v$ to show smoothness via a localization argument.

\pagebreak
\begin{theorem} \label{thm:regThm}
For each fixed $y\in \mathcal{S}$, define $f:\mathbb{R}\times [0,T]\rightarrow \mathbb{R}_+$ by
\[
f(x,t):=\sum_{y'\neq y}\pi[y,y']\,v(x,y',t).
\]
By Theorem \ref{thm:cont}, $v(\cdot,y',\cdot)$ is a $\beta/2$-H\"older continuous function and so is $f$.

Define  $\tilde{L}$ to be the linear operator given by
\[
\tilde{L}:= \frac{1}{2}  a(x,y)^2 \frac{\partial^2}{\partial\,x^2}+\mu(x,y)\frac{\partial}{\partial x} -\frac{\partial}{\partial\,t}-\pi[y],
\]
where $-\pi[y]=\sum_{y'\neq y}\pi[y,y']$ is the rate of leaving $y$.

The function $v$ in (\ref{eq:valueOSPintro}) is the probabilistic solution of the initial-boundary value problem
\[
\begin{array}{rll}
(\mathbb{L}^\pi-\alpha\, )  v(x,y,t) & =0,      & \mbox{in}\; \mathcal{C} \\
                            v(x,y,0) & =g(x),   & \mbox{in}\; \mathbb{R}\times \mathcal{S}\times\{0\} \\
                            v(x,y,t) & = g(x),  & \mbox{on}\; \partial \mathcal{C}
\end{array}
\]t
where $\mathcal{C}=\{(x,y,t)\in \mathbb{R}\times \mathcal{S}\times(0,T]:v(x,y,t)>g(x)\}$, and $\mathbb{L}^\pi$ is given in equation (\ref{bigL}).
In particular, $v(\cdot,y,\cdot) \in C^{2,1}(\mathcal{C})$.
\end{theorem}
\begin{proof}
By definition of $\mathcal{C}$ and $v$, it is clear that $v=g$ on $\partial \mathcal{C}$ and also $v(x,y,0)=g(x)$.

Fix $y\in \mathcal{S}$ and consider the $y$-section $\mathcal{C}_y$ which is an open subset of $\mathbb{R}\times [0,T]$.
Let $R=(x_0,x_1)\times (t_0,t_1)$ be an open and bounded rectangle in $\mathcal{C}_y$.
Now consider the classical initial-boundary value problem
\begin{equation} \label{eq:revisedDirichletFinite}
\begin{aligned}
(\tilde{L}-\alpha )H(x,t) & =-f(x,t),    \;\,\qquad \mbox{in}\; R\cup B^u \\
                 H(x,t)  & = v(x,y,t),          \qquad  \mbox{on}\; \partial R \backslash B^u,
\end{aligned}
\end{equation}
where 
$B^u=(x_0,x_1)\times \{t_1\}$.

Given that both $a$ and $-f$ are uniformly H\"older in $R$, and that the operator $\tilde{L}-\alpha$ is uniformly H\"older parabolic in $R$ (as $a^2(\cdot,y)$ is bounded away from zero in $R$),
there exists a unique solution $H$ to (\ref{eq:revisedDirichletFinite}) which is continuous in $\bar{R}$ and such that
$H \in C^{2,1}(R)$ (see Theorem 6.3.6 in \cite{Friedman1975}).

Recall that $y\in \mathcal{S}$ is fixed from the beginning.
Now, for each $(x,t)\in R$ and $y'\in \mathcal{S}$, extend $H$ as follows:
\[
h(x,y',t):=
\begin{cases}
H(x,t)      & \mbox{if} \, y'=y \\
v(x,y',t)   & \mbox{if} \, y'\neq y.
\end{cases}
\]

\pagebreak
Fix $(x,t)\in R$ and consider the stopping times $T_1=\inf\{s\geq 0: (X_s,t-s)\notin R\}$
and $T_2=\inf\{s\geq 0: Y_s\neq y\}$, and set $\tau=T_1\wedge T_2$.
Notice that $T_1\leq t$ $P_{x,y}$-a.s.
Since $Y_s\equiv y$ up to the time $\tau$, an application of Dynkin's formula yields
\[
\begin{split}
h(x,y,t) & = E_{x,y}\left(e^{-\alpha \tau} h(X_\tau, Y_\tau, t-\tau)\,\right) \\
        & \quad - E_{x,y}\left( \int_0^\tau e^{-\alpha \tau}(\mathbb{L}^\pi-\alpha )h(X_s,Y_s,t-s)ds\,\right).
\end{split}
\]
\indent Now, $h(X_\tau, Y_\tau, t-\tau)=v(X_\tau, Y_\tau, t-\tau)$ $P_{x,y}$-a.s. Indeed, if $\tau=T_1$ we use the boundary condition in (\ref{eq:revisedDirichletFinite}),
otherwise the equality still holds by the definition of $h$.
Moreover, after a simple algebraic manipulation it can be seen that, for each $(x,t)\in R$,
\[
(\mathbb{L}^\pi -\alpha)h(x,y,t)=(\tilde{L}-\alpha)h(x,y,t)+f(x,t)=0.
\]
We then arrive to the expression
\[
h(x,y,t) = E_{x,y}\left(e^{-\alpha \tau} v(X_\tau, Y_\tau, t-\tau)\,\right), \qquad (x,t)\in R.
\]
Given that $\tau$ is bounded above by the first exit time of $(X_s,t-s)$ from the continuation region,
the dynamic programming principle implies that $h(\cdot,y,\cdot)=v(\cdot,y,\cdot)$ in the rectangle $R$.
Therefore $v(\cdot,y,\cdot)\equiv H$ in $R$ and so $v(\cdot,y,\cdot) \in C^{2,1}(R)$ and
\[
(\mathbb{L}^\pi -\alpha)v(x,y,t)=0 \qquad \mbox{in}\; R.
\]
which concludes the proof since $R$ was arbitrary.
\end{proof}


\section{Extremal payoff scenarios} \label{sec:extsce}

In this section, we assume that $a$ and $\mu$ are both in product form and give conditions under which the function $V^l$ in (\ref{eq:infvalue}) is the value function of an optimal stopping problem associated with an extremal jump rate scenario.
To emphasize the dependence on $\pi$, we write $(X^\pi, Y^\pi)$ and $v(x,y,t;\pi)$ instead of $(X,Y)$ and $v(x,y,t)$, respectively.
Roughly, one expects that increasing the variance of $X^\pi$ (controlled by $\pi$) expedites the time at which $X^\pi$ reaches the high values of $g$, hence increasing the payoff since the penalty for the elapsed time via the discount factor is smaller. This intuition leads to the natural candidate strategy: to choose minimal (resp. maximal) variance to achieve the infimum (resp. supremum).

We make the following standing assumption

\medskip
{\bf (A2)}
$
E_{x,y} \left(\sup_{0\leq t\leq T} [e^{-\alpha \tau}g(X^\pi_t)] \right) <\infty.
$

\medskip
Notice that if $g$ has polynomial growth and the coefficients $a$ and $\mu$ have linear growth 
then (A2) follows by results on estimates of the moments of regime-switching diffusions (see Appendix \ref{app:estimates}). 
In particular, the expectation is uniformly bounded over all $\pi$ because it does not depend on the transition rates. 

We now consider the special case 
\begin{equation}\label{dyn}
X_t=x+\int_0^t a(X_s)\sigma(Y_s)dB_s+\int_0^t r(X_s)\mu(Y_s)ds
\end{equation}
and assume that $r$ is positive.
Before we state the main result of this section we need some preliminary results and definitions.
\begin{definition}
We say that transition rates $q_{i,j}$ on $\cS$ are {\em well-ordered} if
\begin{eqnarray}\label{order}
\begin{aligned}
&\sum_{w\geq z}q_{i,w}\leq \sum_{w\geq z}q_{j,w}\text{ for any }i<j<z\\
\text{and}&\\
&\sum_{w\leq z}q_{i,w}\geq \sum_{w\leq z}q_{j,w}\text{ for any }z<i<j
\end{aligned}
\end{eqnarray}
\end{definition}
\begin{definition}
We say that a coupling of $Y$ and $Y'$, copies of a Markov chain on $\cS$  with $Y_0=y<Y_0=y'$, is an {\em ordered coupling} if  $Y'_t\geq Y_t$ for all $t$ a.s.
\end{definition}
 The following theorem, although fairly obvious, may be of some independent interest.
\begin{theorem}
Suppose that $(q_{i,j})_{(i,j)\in \cS\times\cS}$ is a $Q$-matrix on $\cS$, then
we can find an ordered coupling on $\cS$ with transition rates $q$ for every $y<y'$ if and only if the transition rates $q_{i,j}$ are well-ordered.
\end{theorem}
\begin{proof}
Let the coupling measure be $\P$. Consider $Y$ and $Y'$.
Since they are ordered,
\begin{equation}\label{triv}
\P(Y_t\geq z)\leq \P(Y'_t\geq z)\text{ for all }t.
\end{equation}
Dividing both sides of (\ref{triv}) by $t$ and letting $t\rightarrow 0$, we obtain for $z>y'$:
$$
\sum_{w\geq z}q_{y,w}\leq \sum_{w\geq z}q_{y',w},
$$
and this is true for every $y<y'$. 
A similar argument gives the second inequality in (\ref{order}).

\pagebreak
Conversely, suppose that the transition rates $q_{i,j}$ satisfy (\ref{order}), we sketch a standard argument based on thinning Poisson processes.
Define independent Poisson processes $L$ and $M$ with suitably large rates $\lambda$ and $\mu$ respectively. Define an  independent sequence of independent, identically distributed  Uniform[0,1] random variables $(U_n)_{n\geq 1}$; these will be used for randomisation. We use jumps of $L$ to generate (possible) up-jumps of the copies of the Markov chain and jumps of $M$ to generate down-jumps.

Define the (joint) process $\Y$ as follows. First set
$$
Y^i_0=i \text{ for }i\in \cS.
$$
Then let $\Y$ be constant until the first jump time $J_1$ of $L+M$. Now, if the first jump is a jump of $L$, define $Y^i_{J_1}$ by
$Y^i_{J_1}\geq i$ and, for $k>i$
$$
Y^i_{J_1}\geq k \text{ if and only if } U_1\geq 1-\frac{\sum_{w\geq k}q_{i,w}}{\lambda}.
$$
Similarly, if the first jump is a jump of $M$, define $Y^i_{J_1}$ by
$Y^i_{J_1}\leq i$ and, for $k<i$
$$
Y^i_{J_1}\leq k \text{ if and only if } U_1\leq \frac{\sum_{w\geq k}q_{i,w}}{\mu}.
$$
Now proceed recursively, using the successive jumps of $L+M$ and the successive randomising $U_n$.
It is clear that the $i$th component of $\Y$ is a copy of the MC started at $i$. Then a quick comparison of the jump constructions shows that the components of $\Y$ remain ordered thanks to (\ref{order}) and an inductive argument on the chain at jump times of $L+M$.
\end{proof}
\begin{remark}
Note that any skip-free chain, i.e. one where $q_{i,j}=0$ for $|i-j|>1$, is well-ordered.
\end{remark}
\begin{definition}
For any $Q$-matrix, $q$, on $\cS$ define $q^\sigma$ by
$$
q^\sigma_{i,j}=\frac{q_{i,j}}{\sigma^2(i)}.
$$
\end{definition}
\begin{remark}
If $Y$ is a Markov chain with $Q$-matrix $\pi$, then $\pi^\sigma$ is the $Q$-matrix for $\tilde Y$, the Markov chain obtained by time-changing $Y$ using the additive functional $A$ given by $A_t=\int_0^t \sigma^2(Y_s)ds$.
\end{remark}

\pagebreak

\begin{theorem} \label{main}
Define $\pi^l$ by
\begin{equation} \label{eq:optimalpi}
\pi^l[y,y+z]= \inf A^+_{y,z}
\quad
\pi^l[y,y-z]= \sup A^-_{y,z}.
\end{equation}
Suppose that $\sigma(\cdot)$ is monotone increasing, and
$(\pi^l)^\sigma$ is well-ordered, then in the following three cases,
\begin{enumerate}
\item $\mu=0$;
\item $\mu/\sigma^2$ is decreasing and $g$ is decreasing;
\item $\mu/\sigma^2$ is increasing and $g$ is increasing;
\end{enumerate}
the constant rate matrix $\pi^l$ attains the infimum in (\ref{eq:infvalue}).
\end{theorem}
\begin{remark}
Theorem \ref{main} covers the case of pricing American put options with stochas-\\tic/regime-switching volatility. 
Here $\mu=1$, $r(x)=\alpha x>0$ and it is normally assumed that $Y$ is skip-free. Thus $\sigma^2$ increasing is enough for the result to hold.
\end{remark}
\begin{lemma} \label{eq:crucialMon}
Suppose that $\pi$ is such that $\pi^\sigma=(\pi^\sigma[i,j])$, $i,j\in \mathcal{S}=\{1,2,\ldots,m\}$ is well-ordered, then in case 1,2 and 3 of Theorem \ref{main} above
$v(x,\cdot,t;\pi)$ is also increasing on $\mathcal{S}$, for each $(x,t)\in\mathbb{R}\times [0,T]$.
\end{lemma}
\begin{proof}
Suppose that  $\sigma(\cdot)$ is increasing. In case 1, the result is proved in exactly the same way as Theorem 2.5 by Assing et al. \cite{AJO} with $Y$ replaced by $\sigma(Y)$. There it is insisted that $\pi^l$ is skip-free, but only stochastic monotonicity of $Y$ is actually used.
In case 2 a very similar, but extended argument can be used, time-changing away the $Y$ dependence in the diffusion term in the SDE for $X$
to
$$
d\tilde X_t=a(\tilde X_t)dW_t+r(\tilde X_t)\frac{\mu(\tilde Y_t)}{\sigma^2(\tilde Y_t)}
$$
 and deducing by uniqueness in law that the two time-changed solutions, $(\tilde X,\tilde Y)$ and $(\tilde X',\tilde Y')$, to the dynamics started
at $x,y$ and $x',y'$ with $y\leq y'$ and $x\geq x'$ have components which can be ordered:
$$\tilde X'\leq \tilde X\text{ and }\tilde Y'\geq \tilde Y.
$$
A similar argument with ordering on $X$ and $X'$ reversed, works in case 3.

\end{proof}


Since $\pi\in\mathcal{A}$, the generator $\mathbb{L}^\pi$ takes the form
\begin{equation}\label{eq:infgenRS}
\begin{split}
\mathbb{L}^\pi w(x,y,t) & = \frac{1}{2} a^2(x)\, \sigma^2(y) w_{xx}(x,y,t)+r(x) -w_t(x,y,t) \\
                        & \quad + \sum_{z=1}^{m-y}[w(x,y+z,t)-w(x,y,t)]\pi[y,y+z] \\
                        & \quad +\sum_{z=1}^{y-1}[w(x,y-z,t)-w(x,y,t)]\pi[y,y-z].
\end{split}
\end{equation}

The following verification result gives sufficient conditions for a suitable function $w$ to be a lower bound for $V^l$.

\begin{proposition}{\bf (Lower bounds on $V^l$)} \label{prop:verificationMCFinite}
Suppose that $w:\mathbb{R}\times \mathcal{S}\times[0,T] \rightarrow \mathbb{R}$ is a function such that
for each $y\in \mathcal{S}$, $w(\cdot,y,\cdot)$ is continuous, and 
the restriction of $w(\cdot,y,\cdot)$ 
on the open set $\{(x,t)\in \mathbb{R}\times[0,T]: \, w(x,y,t)>g(x)\}$ is $C^{2,1}$.
Suppose that $w$ satisfies:
\begin{equation} \label{eq:characterizeDirDiffFinite}
\begin{aligned}
\inf_{\pi}\,(\mathbb{L}^\pi- \alpha \,)w(x,y,t)           & =0\; \quad \qquad \mbox{in $\mathcal{C}$,} \\
w(x,y,0)  & =g(x)\qquad \mbox{on $\mathbb{R}\times \mathcal{S}\times\{0\}$,} \\
w(x,y,t)  & =g(x)\qquad \mbox{on $\mathbb{R}\times \mathcal{S}\times(0,T] \backslash \mathcal{C}$,}
\end{aligned}
\end{equation}
where $\mathcal{C}=\{(x,y,t)\in \mathbb{R}\times \mathcal{S}\times[0,T]: \, w(x,y,t)>g(x)\}$
and the infimum is taken over all constant and admissible rate matrices.
Then, for each initial condition $(x,y,u)$,
\begin{equation}\label{eq:lowerbdV}
w(x,y,u)\leq V^l(x,y,u).
\end{equation}
\end{proposition}
\begin{proof}
If $u=0$ then $V^l(x,y,0)=g(x)=w(x,y,0)$. Fix an initial condition $(x,y,u)\in \mathbb{R}\times\mathcal{S}\times (0,T]$.
Pick an arbitrary $\pi\in \mathcal{A}$ and
define
the process $N(\pi)=(N_t(\pi))_{t\geq 0}$
by
\[
N_t(\pi):= e^{-\alpha t} w(X^\pi_{t},Y^\pi_{t},u-t), \qquad 0\leq t\leq u
\]

Let $\hat{\tau}\equiv \hat{\tau}^{\pi}:=\inf\{t\geq 0: (X_t^{\pi}, Y^{\pi}_t,u-t)\notin \mathcal{C}\}\leq u$ and,
for each $R>0$, let $U_R\subset \mathbb{R}^2$ be an open ball centered at $(x,y)$ of radius $R$. Let $\tau_R$ be given by
\[
\tau_R=\min\{\hat{\tau},\, \inf\{t\geq 0: \, (X_t^\pi,Y_t^\pi)\notin U_R\}\,\}.
\]
Notice that $\tau_R\rightarrow \hat{\tau}$ almost surely as $R\rightarrow \infty$.

Since $w$ is sufficiently smooth in $\mathcal{C}$,
we can apply It\^o's formula for semimartingales (see Theorem II.33 in \cite{Protter})
to obtain, for each $0\leq t\leq u$,
\[
N_{t\wedge \tau_R}(\pi) - w(x,y,u)
=\int_{0}^{t\wedge \tau_R} e^{-\alpha u}(L^\pi_s - \alpha )w(X^\pi_{s},Y^\pi_{s},t-s)ds + M_{t\wedge \tau_R},
\]
where $M_t=\int_0^{t} e^{-\alpha s} w_x(X^\pi_{s},Y^\pi_{s},t-s) \,a(X_{s}^\pi) \sigma(Y^\pi_{s}) dB_s$
and $L^\pi_s$ is $\mathbb{L}^{\pi_s}$ as in (\ref{eq:infgenRS}).

Given that $w_x(\cdot,y,\cdot), w_y(\cdot,y,\cdot)$ and $a(\cdot)$ are continuous in $\mathcal{C}$, and $U_R$ is a bounded domain, it follows that
$M_{t\wedge \tau_R}$ has bounded quadratic variation for each $t\geq 0$, and hence the process $M_{\cdot \wedge \tau_R}$ is a true martingale.
Moreover, by (\ref{eq:characterizeDirDiffFinite}) we have that
\[
(L_s^\pi - \alpha\,)w(x,y,t-s) \geq 0 \qquad \forall\, (x,y,t-s)\in \mathcal{C},
\]
which yields, for each $R>0$,
$N_{t\wedge \tau_R}(\pi) - w(x,y,u) \geq M_{t \wedge \tau_R}$.
After taking expectation we obtain
\begin{equation}  \label{eq:crucialIneqMin}
w(x,y,u)\leq E_{x,y}( N_{t\wedge \tau_R}(\pi)).
\end{equation}
Using that $\tau_R \rightarrow \hat{\tau}$ and the boundary conditions in (\ref{eq:characterizeDirDiffFinite}),
we obtain the limit
\[
\lim_{t,R\rightarrow \infty} N_{t\wedge \tau_R}(\pi)
=e^{-\alpha \hat{\tau}} w(X^\pi_{\hat{\tau}},Y^\pi_{\hat{\tau}},u-\hat{\tau})
=e^{-\alpha \hat{\tau}} g(X^\pi_{\hat{\tau}}), \quad \mbox{a.s.} 
\]
Thus by dominated convergence (recall that (A2) is assumed), after taking the limit as $R\rightarrow \infty$ and $t\rightarrow \infty$ in (\ref{eq:crucialIneqMin}), we obtain
\begin{equation} \label{eq:minpi}
w(x,y,u)\leq E_{x,y}(e^{-\alpha \hat{\tau}} g(X^\pi_{\hat{\tau}}))
\leq \sup_{\tau \leq u}E_{x,y}(e^{-\alpha \tau} g(X^\pi_{\tau})),
\end{equation}
and this is true for each $\pi \in \mathcal{A}$.
Therefore
\[
w(x,y,u)\leq \inf_{\pi\in \mathcal{A}} \sup_{\tau \leq u}E_{x,y}(e^{-\alpha \tau} g(X^\pi_{\tau})) \, \equiv V^l(x,y,u)
\]
and the proof is complete.
\end{proof}

\noindent {\bf Proof of Theorem \ref{main}}.
Suppose that $\sigma$ is increasing.
First note that
\[
v(x,y,t;\pi^l)\equiv \sup_{\tau \leq t} E_{x,y} (e^{-\alpha \tau} g(X^{\pi^l}_\tau))
\geq \inf_{\pi\in\mathcal{A}} \sup_{\tau \leq t} E_{x,y}(e^{-\alpha \tau} g(X_\tau^\pi)),
\]
so that it remains to show the reverse inequality. 
By Lemma \ref{eq:crucialMon}, $v(x,\cdot,t;\pi^l)$ is increasing and so we have that
\begin{equation} \label{eq:extremes}
\begin{aligned}
\argmin_{\lambda \in A^+_{y,z}}\;[v(x, y+z,t;\pi^l)- v(x,y,t;\pi^l)]\,\lambda   & =\inf\, A^+_{y,z}=\pi^l[y,y+z], \\
\argmin_{\mu \in A^-_{y,z}}\; [v(x,y-z,t;\pi^l)- v(x,y,t;\pi^l)]\,\mu           &=\sup\, A_{y,z}^-  =\pi^l[y,y-z]
\end{aligned}
\end{equation}
so that the infimum in (\ref{eq:characterizeDirDiffFinite}) is attained at $\pi=\pi^l$.
This fact and Theorem \ref{thm:regThm} together imply that $w(x,y,t)=v(x,y,t;\pi^l)$ satisfies the system in (\ref{eq:characterizeDirDiffFinite}). 
Therefore all the conditions of Proposition \ref{prop:verificationMCFinite} are fulfilled by $w(x,y,t)=v(x,y,t;\pi^l)$ and the proof is complete. $\Box$

Reversing the labelling of the states in $\cS$ yields:

\begin{corollary} \label{cor:decresing}
Suppose that $\sigma(\cdot)$ is monotone decreasing.
and
$(\pi^s)^\sigma$ is well-ordered, where
\[
\pi_s^l[y,y+z]= \sup A^+_{y,z}
\quad
\pi_s^l[y,y-z]= \inf A^-_{y,z}
\]
then in the following three cases,
\begin{enumerate}
\item $\mu=0$;
\item $\mu/\sigma^2$ is increasing and $g$ is decreasing;
\item $\mu/\sigma^2$ is decreasing and $g$ is increasing;
\end{enumerate}
Then the result in Theorem \ref{main} remains true.
\end{corollary}

\appendix

\section{Estimates of moments and integrability} \label{app:estimates}
The main goal of this Appendix is to derive some estimates of the moments of the solution to a stochastic differential equation with regime-switching coefficients.

The proof of the proposition below is inspired by ideas in Kyrlov \cite{Krylov}, and it is somewhat an extension of Corollary 2.5.12 in that text. This result is of independent interest, and this is the reason why we assume the following general set-up.

Let $(W_t,\mathcal{F}_t)$ be a $d_1$-dimensional Brownian motion.
Suppose that $y=(y_t)_{t\geq 0}$ is a continuous-time Markov chain, adapted to $(\mathcal{F}_t)_{t\geq 0}$, with finite state space $\mathcal{S}\subset \mathbb{R}$.
The process $r$ determines the regime-switching dynamics.

For $d\in \mathbb{N}$ and $x_0\in \mathbb{R}^d$,
$x=(x_t)_{t\geq 0}$ is a progressively measurable process in $\mathbb{R}^d$, with respect to $(\mathcal{F}_t)_{t\geq 0}$,
satisfying that
\begin{equation} \label{eq:SDE}
x_t=x_0+\int_0^t\sigma_s(x_s,y_s)dW_s + \int_0^t b_s(x_s,y_s)ds, \qquad a.s.
\end{equation}
where $\sigma_t(x,y)$ is a random matrix of dimension $d\times d_1$; and
$b_t(x,y)$ is a random vector of dimension $d$.

The next result corresponds to Corollary 2.5.12 in \cite{Krylov} in the particular case when $\mathcal{S}$ is a singleton.

\begin{proposition} \label{cor:main}
Fix $T>0$, and the initial condition $(x_0,y_0)$. Let there exist a constant $K>0$ such that
\begin{equation} \label{eq:linear_grth}
\|\sigma_t(x,y)\|+|b_t(x,y)|\leq K(1+|x|), \qquad \mbox{for all} \quad t\geq 0, x\in \mathbb{R}^d, y\in \mathcal{S}.
\end{equation}
Then for all $t\in [0,T]$ and $q\geq 0$, there exists a positive constant $N=N(x_0, K, t, q)$ such that
\begin{equation} \label{eq:main_ineq_est}
E\left(\sup_{s\leq t} |x_s|^q \right)     \leq N.
\end{equation}
\end{proposition}
\begin{proof}
Fix an arbitrary $t\in [0,T]$ and $q \geq 0$.
We split the proof into three parts. 

\smallskip
\noindent {\bf (I).} Assume that $x_t(\omega)$ is bounded in $\omega$ and $t$.
Notice that
\[
|x_t|^2  \leq 4\left[ \, |x_0|^2+\left|\int_0^t \sigma_s(x_s,y_s) dW_s \right|^2 + \left| \int_0^t b_s(x_s,y_s)ds \right|^2 \,\right].
\]

The linear growth condition in (\ref{eq:linear_grth}) implies the following. First,
the stochastic integral $M_\cdot= \int_0^\cdot \sigma_s(x_s,y_s) dW_s$ satisfies
\[
E(\langle M \rangle_t) =E\,\left(\int_0^t \|\sigma_s(x_s,y_s)\|^2 ds \right)
\leq 2K^2 E\,\left( \int_0^t (1+|x_s|^2) ds \right)<\infty
\]
for all $t\geq 0$, since $x_t$ is assumed to be bounded. Then $M$ is a martingale.
Second, using H\"older's inequality,
\[
\left| \int_0^t b_s(x_s,y_s)ds \right|^2 \leq t\,\int_0^t |b_s(x_s,y_s)|^2ds \leq  2 K^2 t\,\int_0^t (1+|x_s|^2)ds.
\]
Putting the last assertions together we obtain, after taking supremum over $[0,t]$ and expectation,

\[
\begin{split}
&E \left(\sup_{0\leq s\leq t}  |x_s|^2 \right) \leq 4\left[ \, |x_0|^2+ E \left(\sup_{0\leq s\leq t} |M_s|^2 \right)
    + 2 K^2 t\, E \left(\int_0^t (1+ |x_s|^2)ds \right)\,\right] \\
        & \quad \leq 4 \left[\, |x_0|^2 + 4\, E \,|M_t|^2 + 2 K^2 t\, E \left( \int_0^t (1+ |x_s|^2)ds \right)\, \right] \\
        & \quad \leq 4 \left[\, |x_0|^2 + (2K^2)4 \, E \left(\int_0^t (1+|x_s|^2)ds \right)  + 2 K^2 t\, E \left(\int_0^t (1+ |x_s|^2)ds \right)\, \right] \\
        & \quad \leq 4 |x_0|^2 + 8K^2(4+t) \int_0^t \left(1+E\left( \sup_{0\leq u\leq s} |x_u|^2 \right)\,\right)ds
\end{split}
\]
where we have used Doob's inequality, the fact that $M^2_t-\langle M\rangle_t$ is a martingale
(see for instance \cite[II.1.7 and IV.1.3]{Rev-Yor}), the linear growth condition in (\ref{eq:linear_grth}), Fubini's Theorem and the boundedness of $x_t$.

Now set $\varphi(t)=\sup_{0\leq s\leq t} |x_s|^2$, $a=1+4 |x_0|^2$, and $b=8K^2(4+t)$, so that
\[
1+E(\varphi(t))\leq a +b\int_0^t \{\,1+ E(\varphi(s)) \,\} ds.
\]
Then, by Grownwall's Lemma, we have that $1+E \varphi(t) \leq a\, e^{b t}$, that is
\[
E \left(\sup_{0\leq s\leq t} |x_s|^2 \right) \leq \bar{N}(x_0,K,t)
\]
where $\bar{N}(x_0,K,t)=(1+4 |x_0|^2)e^{8K^2t(4+t)}$.

\smallskip

\noindent {\bf (II).}
Since $x_t$ is continuous and bounded in $t$, it follows that $\sup_{s\leq t} |x_s|^p =(\sup_{s\leq t} |x_s|)^p$ for any $p\geq 0$.
Using this equality with $p=q$ and then with $p=2$,
we obtain that
\[
E \left(\sup_{0\leq s\leq t} |x_s|^q \right)\leq \left(\, E \sup_{0\leq s\leq t} |x_s|^2\,\right)^{q/2}\leq N(x_0,K,t,q)
\]
where we also used H\"older's inequality in the form $E(\eta^q) \leq [E(\eta^2)]^{q/2}$. Here, $N\equiv N(x_0,K,t,q)=\bar{N}(x_0,K,t)^{q/2}$.

\smallskip

\noindent {\bf (III).} We now assume the general case for $x_t(\omega)$.

For each $R>0$, consider the stopping time $\tau_R=\inf\{t\geq 0: |x_t|\geq R\}$. Then the stopped process $x_{t\wedge \tau_R}(\omega)$ is bounded in $\omega,t$ and moreover,

\[
\begin{aligned}
x_{t\wedge \tau_R} & = x_0 + \int_0^{t\wedge \tau_R} \sigma_s(x_s,y_s)dW_s + \int_0^{t\wedge \tau_R} b_s(x_s,y_s)ds \\
                   & = x_0 + \int_0^t I\{s<\tau_R\}\sigma_s(x_{s\wedge \tau_R},y_{s\wedge \tau_R})dW_s \\
                   & \hspace{2cm} + \int_0^t I\{s<\tau_R\}b_s(x_{s\wedge \tau_R},y_{s\wedge \tau_R})ds.
\end{aligned}
\]

Notice that $x_{t\wedge \tau_R}$ solves (\ref{eq:SDE}) only that with the coefficients $\sigma_s(x,y)$,
$b_s(x,y)$ replaced by $I\{s<\tau_R\}\sigma_s(x,y), I\{s<\tau_R\}b_s(x,y)$, respectively.
However, for each fixed $\omega$,
\[
\|I\{s<\tau_R\}\sigma_s(x,y)\| \leq\|\sigma_t(x,y)\|,\quad \mbox{and} \quad |I\{s<\tau_R\}b_t(x,y)|\leq |b_t(x,y)|.
\]
Then the linear growth condition in (\ref{eq:linear_grth}) is satisfied for the coefficients of $x_{t\wedge \tau_R}$.

From parts {\bf (I)-(II)}, we know that
\[
E\left(\sup_{0\leq s\leq t} |x_{s\wedge \tau_R}|^q \right)  \leq   N, \qquad \mbox{for each $R>0$}.
\]
Given that $\lim_{R\rightarrow \infty} \tau_R = \infty$ a.s,
it follows that $\lim_{R\rightarrow \infty} |x_{s\wedge \tau_R}|^{q}= |x_s|^q$ a.s. by continuity of the paths of $x_t$.
As this is true for each $s\leq t$, we must have $|x_s|^q \leq \lim_{R\rightarrow \infty} \sup_{u \leq t} |x_{u\wedge \tau_R}|^{q}$ for each $s\leq t$.
Hence
\[
\sup_{0\leq s\leq t} |x_s|^q \leq \lim_{R\rightarrow \infty} \sup_{0\leq s\leq t}|x_{s\wedge \tau_R}|^{q}, \quad a.s.
\]

Finally, Fatou's Lemma implies
\[
E\left( \sup_{0\leq s\leq t} |x_s|^q \right)
\leq \liminf_{R\rightarrow \infty} E\left( \sup_{0\leq s\leq t}|x_{s\wedge \tau_R}|^{q}\right)
\leq N,
\]
and the proof is complete.
\end{proof}

Notice that the bound $N$ on (\ref{eq:main_ineq_est}) does not depend on the transition rates of the Markov chain $y$.

\section{Proof of Lemma \ref{lemma:lsc}} \label{app:proofs}

\noindent {\bf Proof of Lemma \ref{lemma:lsc}}.
Let us fix $y\in \mathcal{S}$ throughout.
Given an initial condition $x\in \mathbb{R}$, we shall denote by $X^{x}=(X_t^{x})_{t\geq 0}$ the solution to
\begin{equation} \label{eq:sdeFixedPi}
X_t=x+\int_0^t a(X_s,Y_s)dB_s+\int_0^t b(X_s,Y_s)ds, \qquad t\geq 0,\quad Y_0=y.
\end{equation}
We will show that, for every stopping time $\tau\leq t$ with $t\in(0,T]$, the mapping $x\mapsto E\,e^{-\alpha \tau}g(X_\tau^{x})$ is lower semi-continuous. This implies that
\[
x\mapsto \sup_{\tau \leq t} E(e^{-\alpha \tau}g(X_\tau^{x}))\,\equiv v(x,y,t)\qquad \mbox{is lower semi-continuous,}
\]
concluding the proof.

Fix $x\in \mathbb{R}$ and let $N$ be a neighborhood of $x$. For any $x'\in N$, define the stopping times $\tau_R^{x'}=\inf\{t\geq 0:\, |X_t^{x'}|\geq R\}$ for $R>0$.
Let $\tau_R:=\tau_R^x\wedge \tau_R^{x'}$, it follows that
\begin{eqnarray*}
(X_{t\wedge \tau_R}^{x}-X_{t\wedge \tau_R}^{x'})^2 \leq 4\left[\,|x-x'|^2\right.&+\left.\left|\int_0^{t\wedge \tau_R} \{a(X_s^{x},Y_s)-a(X_s^{x'},Y_s)\}dB_s\,\right|^2\right.\\
&+\left. (\int_0^{t\wedge \tau_R}(\mu(X_s^{x},Y_s)-\mu(X_s^{x'},Y_s))ds)^2  \right].
\end{eqnarray*}
Taking expectation on both sides we obtain
\[
E(|X_{t\wedge \tau_R}^{x}-X_{t\wedge \tau_R}^{x'}|^2 )
\leq 4|x-x'|^2+ 4\, D_R^2(1+T)\int_0^t E\,( |X_{s\wedge \tau_R}^{x}-X_{s\wedge \tau_R}^{x'}|^2) ds
\]
where  $D_R>0$ is a Lipschitz constant for $a(\cdot,y)$ and $\mu(\cdot,y)$ uniformly in $y$.
So, Gronwall's Inequality implies
\[
E(|X_{t\wedge \tau_R}^{x}-X_{t\wedge \tau_R}^{x'}|^2)
\leq 4|x-x'|^2 e^{(4\,D_R^2(1+T))\,t},
\]
so it is clear that $X_{t\wedge \tau_R}^{x'}\rightarrow X_{t\wedge \tau_R}^{x}$ in $L^2$-norm as $x\rightarrow x'$, for each $t > 0$ and $R>0$.

Let $\{x_n\}$ be a sequence in $N$ such that $x_n\rightarrow x$. By the previous argument (after passing to a subsequence and relabeling if necessary)
$X_{t\wedge \tau_R}^{x_n}\rightarrow X_{t\wedge \tau_R}^{x}$ a.s. for each $t>0$.
Since the paths of $X^{x_n}$ and $X^{x}$ are continuous and $\tau_R\rightarrow \infty$ as $R\rightarrow \infty$ a.s., we have that
\[
\lim_{n\rightarrow \infty} X_t^{x_n}= X_t^{x},\qquad  \forall\,t>0 \quad a.s.
\]
after letting $R\rightarrow \infty$. Finally, as $g$ is continuous and bounded from below, Fatou's Lemma yields
\[
E(e^{-\alpha \tau} g(X_\tau^{x})) 
\leq \liminf_{n\rightarrow \infty}E(e^{-\alpha \tau}g(X_\tau^{x_n}))
\]
for every stopping time $\tau\leq t$, that is, the mapping $x\mapsto E(e^{-\alpha \tau}g(X_\tau^{x}))$ is lower semi-continuous and the proof is complete. $\Box$

\section*{Acknowledgement}

A. Ocejo's research supported, in part, by funds provided by the University of North Carolina at Charlotte.

\end{document}